\documentclass[]{interact}

\usepackage{epstopdf}
\usepackage[caption=false]{subfig}

\usepackage[natbibapa,nodoi]{apacite}
\usepackage{graphicx}   
\usepackage{fancyhdr}                   
\usepackage{amsmath}                    
\usepackage{amsfonts}
\usepackage{amssymb}
\usepackage{mathrsfs}
\usepackage{stmaryrd}
\usepackage{pst-all,moredefs}
\usepackage{natbib}
\usepackage[T1]{fontenc}   
\usepackage{dsfont}  
\usepackage{enumerate}
\usepackage{indentfirst}
\usepackage{lscape}
\usepackage{rotating}
\usepackage{epsfig,psfrag}
\usepackage{color,multicol}
\usepackage{calc,ifthen}

\theoremstyle{plain}
\newtheorem{theorem}{Theorem}[section]

\newtheorem{proposition}[theorem]{Proposition}

\theoremstyle{definition}

\theoremstyle{remark}
\newtheorem{remark}{Remark}

\def\calh{{\cal H}}

\def\call{{\cal L}}

\def\caly{{\cal Y}}

\def\call{{\cal L}}

\def\hal{{1 \over 2}}

\def\col{\mbox{col}}

\def\L2{{\cal L}_2}
\def\L2e{{\cal L}_{2e}}

\def\rea{\mathbb{R}}

\def\begequarr{\begin{eqnarray}}
\def\endequarr{\end{eqnarray}}
\def\begequarrs{\begin{eqnarray*}}
\def\endequarrs{\end{eqnarray*}}
\def\begarr{\begin{array}}
\def\endarr{\end{array}}
\def\begequ{\begin{equation}}
\def\endequ{\end{equation}}
\def\lab{\label}
\def\begdes{\begin{description}}
\def\enddes{\end{description}}
\def\begenu{\begin{enumerate}}
\def\begite{\begin{itemize}}
\def\endite{\end{itemize}}
\def\endenu{\end{enumerate}}

\def\lef[{\left[\begin{array}}
\def\rig]{\end{array}\right]}
\def\qed{\hfill$\Box \Box \Box$}
\def\begcen{\begin{center}}
\def\endcen{\end{center}}

\def\TAC{{\it IEEE Trans. Automatic Control}}

\def\AUT{{\it Automatica}}

\def\ARC{{\it IFAC Annual Reviews in Control}}

\def\liminf{\lim_{t \to \infty}}

\def\L2e{{\cal L}_{2e}}

\def\rea{\mathds{R}}

\def\adj{\mbox{adj}}
\def\col{\mbox{col}}
\def\hal{{1 \over 2}}

\def\hal{{1 \over 2}}
\def\bfy{{\bf y}}

\usepackage{color}

\begin{document}

\title{Parameter Estimation and Adaptive Control of Euler-Lagrange Systems Using the Power Balance Equation Parameterization}

\author{
\name{Jose Guadalupe Romero\textsuperscript{a}\thanks{ Jose Guadalupe Romero. Email: jose.romerovelazquez@itam.mx}, Romero Ortega\textsuperscript{a}\thanks{Romeo Ortega. Email: romeo.ortega@itam.mx} and  Alexey Bobtsov\textsuperscript{b}\thanks{ Alexey Bobtsov. Email: emmanuel.nuno@cucei.udg.mx}}
\affil{\textsuperscript{a}Departamento   Acad\'emico de Sistemas Digitales, ITAM,  R\'io hondo 1, Progreso Tizap\'an, 01080, Ciudad de M\'exico, M\'exico \\ \textsuperscript{b} Faculty of Control Systems and Robotics, ITMO University, Kronverkskiy av. 49, St. Petersburg, 197101, Russia.}
}

\maketitle

\begin{abstract}                          
It is widely recognized that the existing parameter estimators and adaptive controllers for robot manipulators are extremely complicated to be of practical use. This is mainly due to the fact that the existing parameterization includes the complicated signal and parameter relations introduced by the Coriolis and centrifugal forces matrix. In an insightful remark of their seminal paper Slotine and Li suggested to use the parameterization of the power balance equation, which avoids these terms---yielding significantly simpler designs. To the best of our knowledge, such an approach was never actually pursued in on-line implementations, because the excitation requirements for the consistent estimation of the parameters is ``very high". In this paper we use a recent technique of generation of ``exciting" regressors developed by the authors to overcome this fundamental problem. The result is applied to general Euler-Lagrange systems and the fundamental  advantages of the new parameterization are illustrated with comprehensive simulations of a 2 degrees-of-freedom robot manipulator. 
\end{abstract}

\section{Introduction}

It is well-known that the implementation of on-line parameter identifiers and adaptive controllers for Euler-Lagrange (EL) systems is complex and very computationally demanding  \citep{GAUKHAIJRR, KHADOM,NIESLO, ORTbook, SPOHUTVID}. This is mainly due to the fact that the parameterization that is used to obtain the linear regression equation (LRE) needed for their implementation---introduced in  \citep{KHOKAN}---is based on the full model of the system dynamics,  which involves complicated signal and parameter relations introduced in the Coriolis and centrifugal forces matrix. An additional difficulty for the application of these adaptive techniques is that, in order to obtain a {\em linear} relation in the LRE, it is necessary to {\em overparametrize} the vector of unknown parameters. This approach has very serious shortcomings, in particular the need of more {\em stringent excitation conditions} stemming from the fact that the parameter search takes place in a bigger dimensional space with nonunique minimizing solutions---see  \citep{LJUbook,SASBODbook} and the detailed discussion in  \citep[Section 1]{ORTetalaut21}. This situation has severely stymied the practical implementation of  these advanced identification and control techniques in many critical applications, for instance, for robot manipulators \citep{HUACHI,NIESLO,ZHAWE}. 

A computational complexity reduction is achieved restricting ourselves to the estimation of the so-called {\em base inertial parameters} introduced in \citep{GAUKHACDC}, which exploits the fact that the matrix defining the LRE of  \citep{KHOKAN} is not full rank, see \citep{SOUCOR} for some recent developments of this approach.  Another route, pursued by practitioners, to overcome this difficulty is to replace the complicated expression of the regressor using function approximation, leading to the so-called regressor-free adaptive controllers. See \citep{HUACHI} for a detailed description of this procedure applied to robot manipulators. Unfortunately, as always with function approximation-based techniques \citep{ORT}, although they might lead to successful designs, there is no solid theoretical guarantee that  the procedure will work---see \citep[Section 4.5]{HUACHI}.

In an insightful remark of their seminal paper \citep[Section 2.2]{SLOLIaut} the authors suggested to use the parameterization of the power balance equation. The main advantage of this approach is that, as mentioned above, the resulting LRE avoids the cumbersome terms related to the Coriolis and centrifugal forces matrix. This is a significant simplification that drastically reduces the complexity and computational demands. To the best of our knowledge, such an approach was never actually pursued, because the {\em excitation requirements} for the consistent estimation of the parameters is ``very high"---see \citep[Remark 16]{ORTetalaut21}.  One notable exception where this parameterization was used is in the work of Niemeyer and Slotine, where it was combined with the classical parameterization, in a composite adaptive controller. In an independent line of research, the use of the power balance equation for parameter estimation was also suggested in \citep{GAUKHACDC}, see \citep[Subsection 12.6.2]{KHADOM} for a detailed description of the model. In contrast with the proposal of \citep{SLOLIaut}, where in a standard way a LRE is used for on-line estimation, in the aforementioned papers the power balance equation is integrated in a series of intervals to generate an overdetermined set of linear equations, from which they identify the parameters via least-squares minimization procedures. This approach was also used in \citep{BLOCK} to identify the parameters of the pendubot. Interestingly, the author observed that the identification of the friction terms was very problematic with this method---an issue also discussed in \citep{PRUFERetal}, where the lack of excitation is identified as the culprit of this problem.

 In this paper we propose a procedure to overcome, for the first time, this fundamental problem. Towards this end we use a recent technique of generation of new LRE with ``exciting" regressors developed by the authors in \citep{BOBetal}. The result is applied to general Euler-Lagrange systems and the significant advantages of the new parameterization  of the power balance equation are illustrated with comprehensive simulations of a 2-degrees-of-freedom (DOF) robot manipulator. 

The development of the new LRE of  \citep{BOBetal} relies on the use of the following components: (i) the {\em dynamic regressor extension and mixing} (DREM) estimator  \citep{ARAetaltac},  which is a procedure that generates, from a $q$-dimensional LRE, $q$ scalar LREs, one for each of the unknown parameters; (ii) the {\em  parameter estimation based observer} (PEBO) proposed in  \citep{ORTetalscl}, later generalized in \citep{ORTetalaut}, that translates the problem of {\em state} estimation into a {\em parameter} estimation one; (iii) the energy pumping-and-damping injection principle of \citep{YIetal} to inject excitation to the new regressor. To make the paper self-contained all these derivations are briefly summarized in the Appendix.
For the sake of clarity, we restrict ourselves to the case of LRE. However, as indicated in the concluding remarks, the regressor generator proposed in the paper can be applied to the nonlinearly parameterized case.\\

\noindent {\bf Notation.} $I_n$ is the $n \times n$ identity matrix.  For $x \in \rea^n$ we denote the Euclidean norm as $|x|:=\sqrt {x^\top x}$. $\call_1$ and $\call_2$  denote the  absolute integrable and square integrable function spaces, respectively.

\section{A Power Balance Equation-based Parameterization of EL Systems}
\lab{sec2}
%
Following the insightful remark of  \citep[Section 2.2]{SLOLIaut}, in this section we derive a LRE for general EL systems using the power balance equation and compare its complexity with the ``classical" parameterization using the EL equations of motion, see also \citep[Subsection 12.6.2]{KHADOM}.  
\subsection{System dynamics}
\lab{subsec21}
%
In this paper we consider $n_q$-DOF, underactuated, EL systems with generalized coordinates $q\in \rea^{n_q}$ and control vector $\tau \in \rea^m$, with $m \leq n_q$, whose dynamics is described by the EL equations of motion
\begequ
\lab{elsys}
{ \frac{d}{dt} \left[\nabla_{\dot q}{\cal L}(q,\dot q)\right] - \nabla_q {\cal L}(q,\dot q)= G(q,\dot q) \tau,}
\endequ
where $\call:\rea^{n_q} \times \rea^{n_q} \to \rea$ is the Lagrangian function
$$
{\cal L}(q,\dot q) :=   {T}(q,\dot q) - {U}(q),
$$
with  ${T}:\rea^{n_q} \times \rea^{n_q} \to \rea$ the kinetic co-energy function, ${U}:\rea^{n_q} \to \rea$ the potential energy function, $G:\rea^{n_q} \times \rea^{n_q} \to \rea^{n_q \times m}$ is the input matrix, which is assumed {\em known}. For ${\cal L}$, the transposed gradient, with respect to $q$ and $\dot q$ are denoted by $\nabla_{q} {\cal L}$ and $\nabla_{\dot q} {\cal L}$, respectively. We restrict our attention to {\em simple} EL systems, whose kinetic energy is of the form
$$
{T}(q,\dot q)=\frac{1}{2}\dot q^{\top}M(q)\dot q,
$$
where $M:\rea^{n_q} \to \rea^{n_q \times n_q}$, $M(q) > 0,$ is the generalized inertia matrix. See \citep{ORTbook} for additional details on this model and many practical examples and   \citep{SPOHUTVID} for a detailed description of robot manipulators.

\begin{remark}
\lab{rem1}
It is possible to include in the dynamics \eqref{elsys} the effect of ``linear" friction terms of the form
$$
{ \frac{d}{dt} \left[\nabla_{\dot q}{\cal L}(q,\dot q)\right] - \nabla_q {\cal L}(q,\dot q)= G(q,\dot q) \tau +R\dot q,}
$$
with $R \in \rea^{n_q \times n_q}$ a diagonal, positive semidefinite matrix with {\em unknown} coefficients. As shown in Remark \ref{rem2} below, this effect can also be included in our analysis. However, for the sake of brevity, this additional term is omitted in the sequel.
\end{remark}
\subsection{Derivation of the new regression equation}
\lab{subsec22}
%
A first step in the design of parameter estimators is the derivation of a LRE for the unknown parameters of the EL system. Towards this end, we introduce the following parameterization of the inertia matrix $M(q)$ and the potential energy $U(q)$ 
\begin{equation}
\label{parmu}
M(q) = \sum_{i=1}^\ell   {\bf m}_i(q) \theta_i^M, \hspace{.5cm} {{U}}(q)=\sum_{j=1}^r {{U}}_j(q)  \theta_j^{{U}},
\end{equation}
 with  {\it known} matrices  ${\bf m}_i: \rea^{n_q} \rightarrow  \rea^{n_q\times n_q}$ and functions ${{U}}_j: \rea^{n_q} \rightarrow \rea$ and  $w:=\ell+r$ {\em unknown} parameters $\theta_i^M, \theta_j^U$, that we group together  in a single vector  as
\begequ
\lab{thesthe}
\theta:=\col(\theta_1^M, \cdots, \theta_l^M, \theta_1^U,  \cdots, \theta_r^U) \in \rea^w.
\endequ 
We are in position to present the following.
%
\begin{proposition} \em
\label{pro1}
Define the vector  $\Omega \in \rea^w$ via
\begin{align}
\nonumber
\dot z &= -\lambda(z+\omega)\\
\lab{Omega}
\Omega&=z+\omega
\end{align}
with $\lambda>0$ a design parameter and 
\begin{equation}
\label{omega}
\omega:= \left[ \begin{array}c {1\over 2} \dot q^\top  {\bf m}_1(q) \dot q \\ \vdots  \\ {1\over2} \dot q^\top {\bf m}_\ell(q) \dot  q \\   U_1(q) \\ \vdots \\ U_r(q)  \end{array} \right]   \in \rea^w.
\end{equation}

The EL system \eqref{elsys} satisfies the LRE 
\begin{equation}
\label{newlreel}
y= \Omega^\top  \theta
\end{equation}
where 
\begequ
\lab{yel}
\dot y=-\lambda y + \dot q^\top G\tau,
\endequ
and $\theta$  is defined via \eqref{thesthe}.
\end{proposition}
%
\begin{proof}
As shown in \citep[Proposition 2.5]{ORTbook} EL systems define a  {passive operator}  $G(q,\dot q) \tau \mapsto  \dot q$ with storage function 
\begequ
\lab{enefun}
{\cal E}(q,\dot q):={T}(q,\dot q) + {U}(q).
\endequ
More precisely, it satisfies the  {power balance} equation
\begin{eqnarray}
\dot {\mathcal E}= \dot q^\top G\tau.
\label{dotE}
\end{eqnarray}
Now, applying the LTI filter 
\begin{equation}
\label{fil}
H(p)={1 \over {p+\lambda}}
\end{equation}
where  $p:={d \over dt}$, to both sides of \eqref{dotE} we get
\begin{equation}
\label{newy}
p H(p)[ {\mathcal E}]= y,
\end{equation}
with the filter state realization given in \eqref{yel}. On the other hand, using the parameterization \eqref{parmu}, the energy function \eqref{enefun} can be written as 
 \begin{eqnarray}
 {\mathcal E}(q,\dot q)&=& {1\over 2} \sum_{i=1}^\ell  \dot q^\top {\bf  m}_i(q) \dot q {\theta_i^M} +  \sum_{j=1}^r U_j(q)  {\theta_j^U} \nonumber \\
 &=& \omega^\top {\theta}
 \end{eqnarray}  
where we used \eqref{thesthe} and  \eqref{omega}. The proof is completed noting that \eqref{Omega} is a state realization of the filter equation
$$
\Omega:=p H(p)[\omega].
$$
\end{proof}

\begin{remark}
\lab{rem2}
In relation to Remark \ref{rem1}, the effect of the friction terms in the power balance equation \eqref{dotE} is as follows
$$
\dot {\mathcal E}= \dot q^\top G\tau - \dot q^\top R \dot q.
$$
Hence, the incorporation of the unknown matrix $R$ in the LRE \eqref{newlreel} leads to the new LRE
$$
y= \Omega^\top  \theta+ \Omega_R^\top  \theta_R,
$$
where $\theta_R:=\col(R_1,R_2,\dots,R_{n_q})$ are the diagonal elements of the matrix $R$ and
$$
\Omega_R:=H(p)[\col(\dot q^2_1,\dot q^2_2,\dots,\dot q^2_{n_q})].
$$
\end{remark}

\begin{remark}
\lab{rem3}
It is important to note that the parameters $\theta$ in  \eqref{parmu} are not the {\em physical} parameters of the system. But they are obtained {\em overparameterizing} the truly physical ones to obtain a linear parameterization. This fact will become clear in  the 2-dof example treated below. See also \citep{ORTetalaut21} where a procedure to identify the true physical parameters, using a nonlinear parameterization, is proposed. 
\end{remark}
\subsection{Comparison with the ``classical" parameterization}
\lab{subsec23}
%
The fact that it is possible to use the power balance equation \eqref{dotE} to obtain a parameterization of the robot manipulator dynamics was indicated in an insightful remark in \citep{SLOLIaut}, but was not further elaborated. To the best of the authors' knowledge all results on parameter estimation and adaptive control of this kind of systems have relied on the far more involved parameterization of the full dynamics \eqref{parmu}, first proposed in \citep{KHOKAN} and cleverly exploited in \citep{SLOLItac}, which we briefly review below. The main reason why this parameterization was not used is because of the stringent excitation requirements that it imposes. The main contribution of our paper is to show that using the DREM procedure \citep{ARAetaltac} and the new LRE proposed in  \citep{BOBetal} it is possible to generate alternative LRE---with exciting regressors---to overcome this drawback.

To obtain the ``classical" parameterization it is necessary to write the dynamics of the EL system \eqref{elsys} as
\begin{equation}
\label{EL2}
{d \over dt }\left[  M(q) \dot q \right] - \frac{1}{2} \nabla_q \left[  \dot q^{\top}M(q)\dot q\right]   + \nabla {U}(q) = G(q)\tau.
\end{equation}

We are in position to present the following well known result  \citep{KHOKAN}, which is given here for the sake of completeness.
\begin{proposition} \em
\label{pro6}
The EL system \eqref{EL2} satisfies the vector LRE 
\begin{equation}
\label{clalre}
\bfy= \Psi  \theta,
\end{equation}
with $\theta$ defined in \eqref{parmu} and \eqref{thesthe},
\begequ
\lab{yelcla}
\dot \bfy=- \lambda \bfy+G\tau,
\endequ
and  the ${n_q \times w}$  regressor matrix is defined a
\begin{align}
\lab{clareg}
\dot \Psi= -\lambda \Psi+  \left[ \begin{array}c  p  {\bf m}_1(q)  \dot q -{1\over 2} \nabla_q  (\dot q^\top  {\bf m}_1(q) \dot q ) \\  \vdots  \\ p {\bf m}_\ell(q)\dot q-{1\over2} \nabla_q (\dot q^\top {\bf m}_\ell(q) \dot q) \\ \nabla {{U}}_1(q) \\  \vdots  \\ \nabla {{U}}_r(q) 
 \end{array} \right]^\top. 
\end{align}
\end{proposition}
\begin{proof}
Applying the LTI filter \eqref{fil} to both sides of \eqref{EL2} we get
\begin{equation}
p H(p)[ M(q) \dot q ]-\frac{1}{2}  H(p)\left[ \nabla_q (  \dot q^{\top}M(q)\dot q )\right] +H(p)[ \nabla {{U}}(q)] = \bfy,
\label{newyel}
\end{equation}
where we have used \eqref{yelcla}. Now, using the parameterization \eqref{parmu},  the left hand side of  \eqref{newyel}  can be written as 
\begin{align}
 \Psi  \theta= &\sum_{i=1}^\ell  H(p)  \Big [  p  [{\bf m}_i(q)  \dot q]- \frac{1}{2} \nabla_q (  \dot q^\top  {\bf m}_i(q)  \dot q  )\Big ]\theta_i^M + \nonumber \\ &\sum_{j=1}^r H(p)[  \nabla {{U}}_j(q)]  \theta_j^{{U}}, 
  \label{Ec1}
\end{align}
where we used \eqref{thesthe}.
This completes the proof.
\end{proof} 

\begin{remark}
\lab{rem4}
The reduction of the computational complexity of the new LRE \eqref{newlreel}, with respect to the one of the ``classical" LRE \eqref{clalre}, can hardly be overestimated. It suffices to compare the $w$-dimensional vector regressor  \eqref{omega} with the  ${n_q \times w}$  regressor matrix \eqref{clareg}.
\end{remark}
%
\section{Generation of  ``Exciting" Regressors for the New Parameterization}
\lab{sec3}
%
As mentioned above the main drawback of the new parameterization  \eqref{Omega}-\eqref{yel} is that the excitation requirements for consistent estimation are very ``high". In this section we apply the procedure proposed in \citep{BOBetal} to generate new LRE where the regressor has ``improved" excitation properties. \\

The generation of new scalar LREs proceeds along the following steps.

\begenu[{\bf S1}]
\item  Apply DREM to the original LRE to generate {\em scalar} LREs---one for each one of the parameters to be estimated. This procedure is summarized in Proposition \ref{pro4} in Appendix A, where the dynamic extension is done following Kreisselmeier's suggestion  \citep{KRE}.\footnote{See  \citep{ORTNIKGER,ORTetaltac} for other methods to construct the dynamic extension.}
\item Construct the new LRE folllowing the procedure proposed in  point {\bf P1} of Proposition \ref{pro5} in Appendix A.
\item Select the ``input" signals $\boldsymbol u_i$ of the LRE generator as suggested in point {\bf P2} of Proposition \ref{pro5} in Appendix A.
\endenu

Once the new LREs have been generated the estimator design is completed applying---for instance---a simple gradient descent-based parameter adaptation algorithm like the one suggested in Proposition \ref{pro6} in Appendix A. 

In summary, the proposed estimator with the new LRE proceeds from the original LRE---that is the one obtained from the power balance equation \eqref{newlreel} or the classical LRE \eqref{clalre}---then implements the dynamic extension  \eqref{eq1},  \eqref{eq2} and \eqref{eq3} and wraps-up the design with the gradient estimator  \eqref{eq4}.

The following remarks pertaining to the convergence properties of the estimator based on the new LRE are in order---see \citep{BOBetal} for additional details.
\begenu
\item[{\bf R1}] The condition \eqref{conalpdel} can, in principle, be easily satisfied with a suitable definition of the function $\boldsymbol \alpha(t)$. However, this restriction on this design parameter will, in general, limit our ability to generate an ``exciting" regressor $\Phi_{21}$ for the new LRE \eqref{newlre}, affecting the performance of the estimator.
\item[{\bf R2}]  Point {\bf P2} of Proposition \ref{pro6} guarantees the asymptotic convergence of the estimator, provided the condition $\liminf \Phi_{11}(t) \neq \sqrt{2\beta}$ is satisfied. As explained in \citep{BOBetal}, this is a technical condition needed to ensure the boundedness of all the signals, and it is ``generically true". 
\endenu
%
\section{Indirect Adaptive Control}
\lab{sec4}
%
In this section we combine---in a certainty-equivalent way---the parameter estimator proposed in the previous section with a globally stabilizing controller. 

The formulation of the adaptive control problem requires  the following stabilizability condition for the case of known parameters.

{\bf Assumption 1}
Given a desired bounded trajectory for the state vector $(q_\star(t), \dot q_\star(t)) \in \rea^{n_q}\times \rea^{n_q}$. Define the state tracking error $\col(\tilde q,\dot {\tilde q}):=\col(q-q_\star,\dot q - \dot q_\star).$ There exists a mapping $\beta: \rea^{n_q} \times \rea^{n_q} \times \rea^w \times \rea_{\ge 0} \to \rea^m$, such that the system
$$
M(q) \ddot q+ C(q, \dot q)\dot q + \nabla {{U}}(q) =G(q) \beta(q,\dot q,\theta,t),
$$
has an {\em error dynamics} 
$$
 \left[ \begin{array}{c} \dot {\tilde q}   \\
    \ddot {\tilde q} \end{array}\right]=f_\star(\tilde q, \dot {\tilde q},\theta,t)
$$
with $f_\star: \rea^{n_q} \times \rea^{n_q} \times \rea^w \times \rea_{\ge 0} \to \rea^{2n_q}$, whose origin is globally exponentially stable.  
\\

The {\em control objective} is then to design a parameter estimator such that the (certainty-equivalent) adaptive control $\tau=\beta(q,\dot q,\hat \theta,t)$ ensures global asymptotic tracking, that is, 
\begequ
\lab{asyconel}
\liminf \col(\tilde q(t),\dot {\tilde q}(t)) = 0,
\endequ
with all signals bounded. 

We are in position to present the main result of this section. The proof exploits the fact that we have {\em consistent estimates}, and it follows {\em verbatim} the proof of \citep[Proposition 6]{ORTetalaut}---see also  \citep[Proposition 8]{ORTetalaut}---hence, it is omitted for brevity.

\begin{proposition} \em
\lab{pro3}
Consider the EL system \eqref{EL2} verifying Assumption {\bf 1} in closed-loop with the certainty-equivalent adaptive control 
 \begin{equation}
 \label{con1}
 \tau=\beta(q,\dot q, \hat \theta,t) , 
 \end{equation}
where the estimated parameters are generated, proceeding from the LRE  \eqref{newlreel}, as suggested in the previous section.   Assume the conditions of Propositions \ref{pro5} and  \ref{pro6} (that ensure a consistent estimation) hold.  Under these conditions,  \eqref{asyconel} holds with all signals bounded.
\end{proposition} 

\begin{remark}
\lab{rem5} 
For {\em fully actuated} systems, {\em i.e.}, $m=n_q$, the mapping $ \beta(q,\dot q,\theta,t)$ can be chosen  as the {\em Slotine-Li Controller} that, in the known parameter case, it is given by \citep{SLOLItac}
\begequ
\lab{slolicon}
\beta(q,\dot q,\theta,t)= M(q) \ddot q_r + C(q,\dot q) \dot q_r + g(q) +K_1 s,
\endequ
where we defined the signals
$$
\dot q_r := \dot q_\star  - K_2 \tilde q,\quad s := \dot {\tilde q} + K_2 \tilde q,
$$
with diagonal, positive definite gains  $K_1$, $K_2$  $\in \rea^{n_q \times n_q}$ . The closed-loop system is then
$$
	M(q)\dot s +[ C(q,\dot q) + K_1] s = 0,\;\dot {\tilde q}+ K_2 \tilde q=   s,
$$
that---as indicated in \citep[Remark 4.5]{ORTbook}---has a globally exponentially stable equilibrium at the origin. 
\end{remark}
%

\section{Application to a Fully Actuated $2$-dof Robot Manipulator}
\lab{sec5}

In this section we present the two LRE derived above  for the  classical fully actuated 2-dof robot manipulator \citep{Craig}. 
\subsection{Derivation of the new   \eqref{newlreel} and classical   \eqref{clalre} LREs}
\lab{subsec51}
%

The equation of motion of the robot is given by \eqref{EL2} with  
\begin{eqnarray}
\label{MU}
M(q)=&\left[ \begin{array}{cc} \theta_1 +2 \theta_2 \cos(q_2) & \theta_3 +\theta_2 \cos(q_2)  \\ \theta_3 +\theta_2\cos(q_2)   &  \theta_3
      \end{array}\right] \nonumber  \\
       U(q)=&\theta_4 g \left(1+\sin(q_1+q_2)\right) + \theta_5 g \left(1+\sin(q_1)\right),  
\end{eqnarray}
$G=I_2$, with $g$ the gravitational constant and the unknown parameters 
$$
\theta=\left[ \begin{array}{c} l_2^2 m_2 + l_1^2 (m_1+m_2)\\l_1 l_2 m_2\\l_2^2 m_2\\l_2 m_2\\l_1 (m_1+m_2) \end{array} \right],
$$
where $l_i$ is the the length of the link $i$ with mass $m_i$ for $i=1,2$. Now, following \eqref{parmu} we define 
\begin{align}
{\bf m}_1 & :=\left[ \begin{array}{cc} 1 & 0 \\ 0& 0 \end{array} \right], {\bf m}_2(q_2):=\cos(q_2)\left[ \begin{array}{cc}  2  & 1 \\ 1& 0  \end{array} \right] \\{\bf m}_3&:= \left[ \begin{array}{cc}  0 & 1 \\ 1& 1 \end{array} \right], \;
U_1(q):=g[1+\sin(q_1+q_2)]\\U_2(q_1)&:=g[1+\sin(q_1)].
\end{align}
Thus, the LRE  \eqref{newlreel} holds with 
$$
y=H(p)[\dot q^\top \tau],
$$
and the regressor vector
\begin{align}
\Omega=pH(p)\left[ \begin{array}{c} {1\over 2} {\dot q}_1^2\\({\dot q}_1^2 +{\dot q}_1 {\dot q}_2)\cos(q_2) \\{1\over 2} {\dot q}_2^2 +{\dot q}_1 {\dot q}_2\\g(1+\sin(q_1 +q_2))\\g(1+ \sin(q_1)) \end{array} \right].
\end{align}

On the other hand, the classical LRE \eqref{clalre} is given 
by $\bfy=H(p) \tau$ and  
$$
\Psi=H(p) \left[ \begin{array}{ll}
p \dot q_1 & 0\\
p \cos(q_2) (2\dot q_1 + \dot q_2) &  \left.\begin{array}{l} p \cos(q_2) \dot q_1+ \\ \sin(q_2) (\dot q^2_1 + \dot q_1 \dot q_2)  \end{array}\right.\\
p\dot q_2 & p(\dot q_1 + \dot q_2)\\
g \cos(q_1+q_2) & g \cos(q_1+q_2)\\
g\cos(q_1) & 0
\end{array} \hspace{-.2cm}\right]^\top.
$$
\subsection{Comparative simulation results}
\lab{subsec52}
%
In this section, we present comparative simulations of the parameter estimators and adaptive control  using the classical   \eqref{clalre} and the new parameterization  \eqref{newlreel}, with and without DREM and using the new LRE or not. First, we give the results for open-loop parameter identification for different input signals. Then, we present the ones obtained for the Slotine-Li  adaptive controller. 

The unknown parameters of the robot are taken as $l_1=0.7$m; $l_2=0.8$m; $m_1=1.5$kg; and $m_2=0.5$kg. The filter \eqref{fil} is implemented with $\lambda=1$ and zero initial conditions. For all simulations, the initial estimates is $\hat\theta_i(0)=0$,   initial velocities $\dot q (0)=\col(0, 0)$ and initial positions  $q(0)=\col(0.6\pi,0.7\pi)$rad.

For the new LRE generator of Proposition \ref{pro4}, we  set $ \beta = \frac{1}{4}$, and the free function $\alpha$ is selected as
$
\alpha(t) = \sin\Big(\frac{1}{5} t\Big).
$

\subsubsection*{Open loop parameter estimation}
\label{subsubsec521}
To evaluate the effect of the richness content of the input signal on the estimator performance we consider the following signals $\tau$.  
\begin{eqnarray}
\begin{aligned}
    \tau_a(t) & = \col \left(e^{-0.4t},\ e^{-0.5t}\right)\\
    \tau_b(t) & =
    \left\{
    \begin{aligned}
         \col \left(1, \, 3 \right) & &t\in[0,2] \mbox{s}
         \\
         0 & & t> 2\mbox{s}
    \end{aligned}
    \right. \\
    \tau_c(t) & = {\cos(4t)\over 2 + t} \col \left(1, \, 3 \right). 
\end{aligned}
\label{taus}
\end{eqnarray}
Notice that these signals are  interval exciting, but clearly not persistently exciting (PE). Therefore, there is no guarantee that the  standard gradient estimator for the classical LRE \eqref{clalre}, which is given by 
\begequ
\lab{paa}
\dot{\hat{\theta}} =  \Gamma  \Psi^\top  \left({\bf y}  - \Psi \hat \theta\right), 
\endequ
with $\Gamma=\Gamma^\top \in \rea^{5\times 5}$, positive definite, will ensure parameter convergence. Actually, we show in the simulations that this estimator has in all cases a steady-state error. On the other hand, we show that using  DREM the estimated parameters converge but after a certain time a drift away from the true value is observed for all three input signals. The latter undesired behavior, which is probably due to the loss of excitation and the accumulation of numerical integration errors, is avoided if we additionally use the new LRE. In these simulations we set $\Gamma=25I_5$ for the estimator \eqref{paa} and $\gamma_i=25$ for the DREM ones.

The result of these simulations is presented in   Figs \ref{fig1}-\ref{fig3}, where we also show the behavior of the regressor signals $\Delta$ and $\Phi_{21}$.   As seen from the figures, in  the three cases, $\Delta(t) \to 0$ loosing excitation, while $\Phi_{21}(t) \not \to 0$, guaranteeing PE and consequently exponential convergence. Moreover, we show in the figures the regressors $\Phi_{21_i}$ of the five elements of the vector theta, but the plots are almost overlapped, therefore, hardly distinguishable.
This behavior explains the long term drift of DREM with the original LRE, which is avoided by the use of the new LRE.

\begin{figure}[htp]
\centering
  \includegraphics[width=0.81\textwidth]{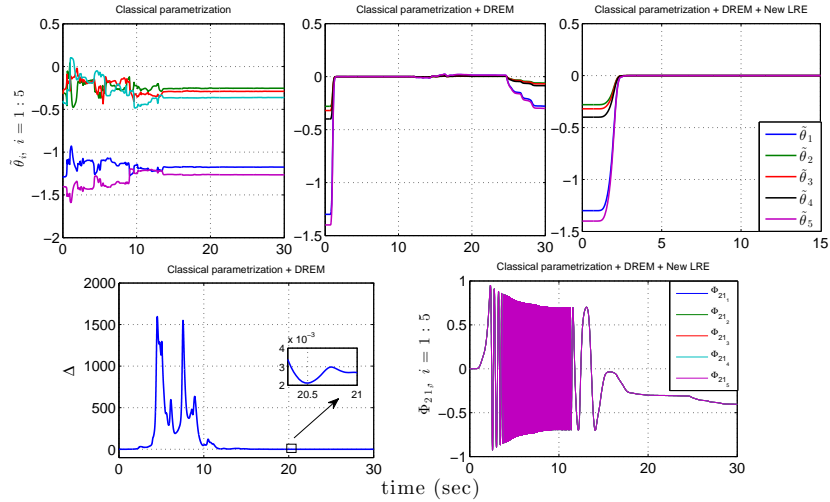}  
  \caption{Transient behavior of the signals $\tilde \theta_i$, $\Delta$ and  $\Phi_{21_i}$ using $\tau_a$.}
 \label{fig1}
\end{figure}
\begin{figure}[htp]
\centering
  \includegraphics[width=.81\textwidth]{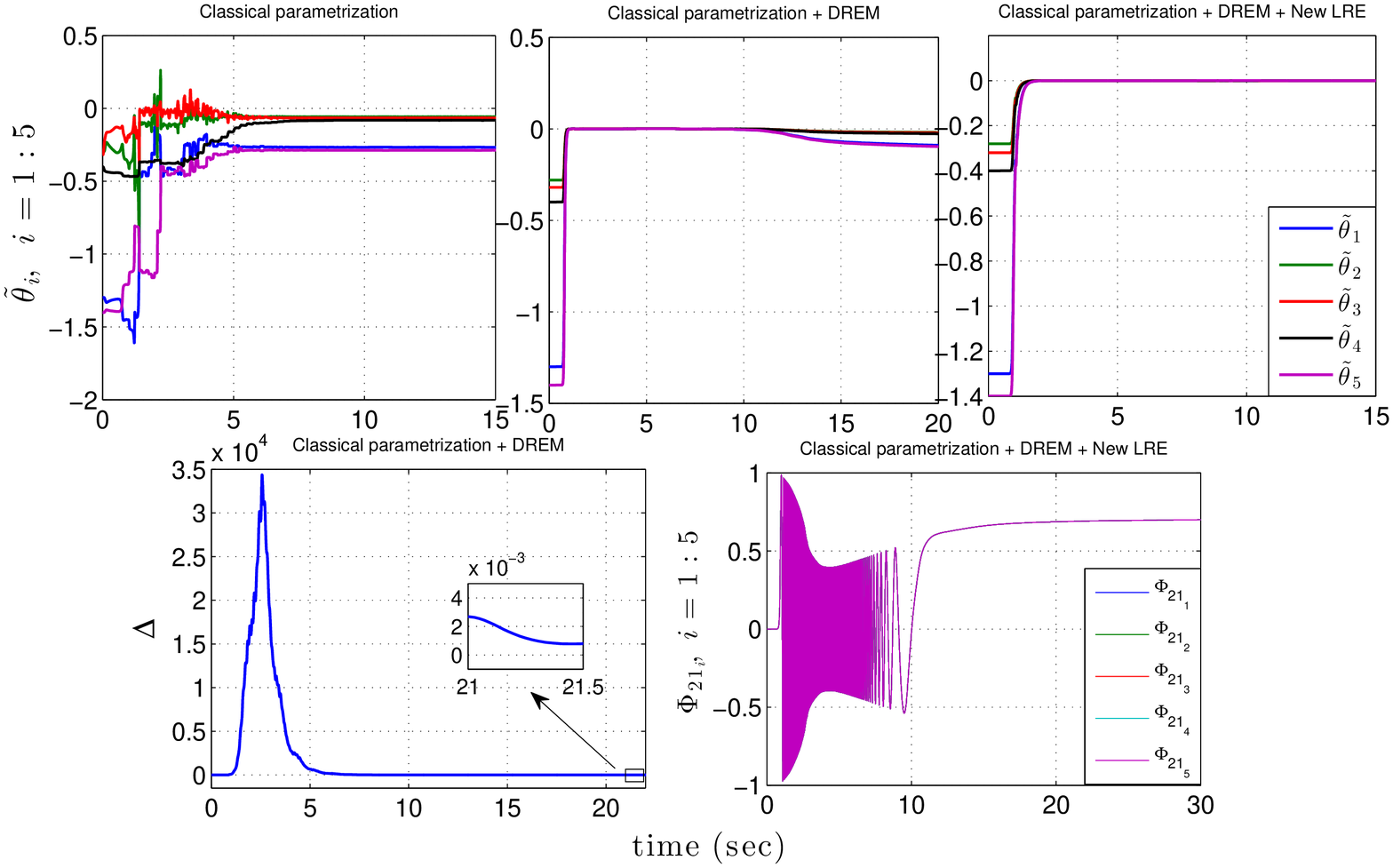}  
  \caption{Transient behavior of the signals $\tilde \theta_i$ using $\tau_b$.}
  \label{fig2}
\end{figure}
\begin{figure}[htp]
\centering
  \includegraphics[width=.81\textwidth]{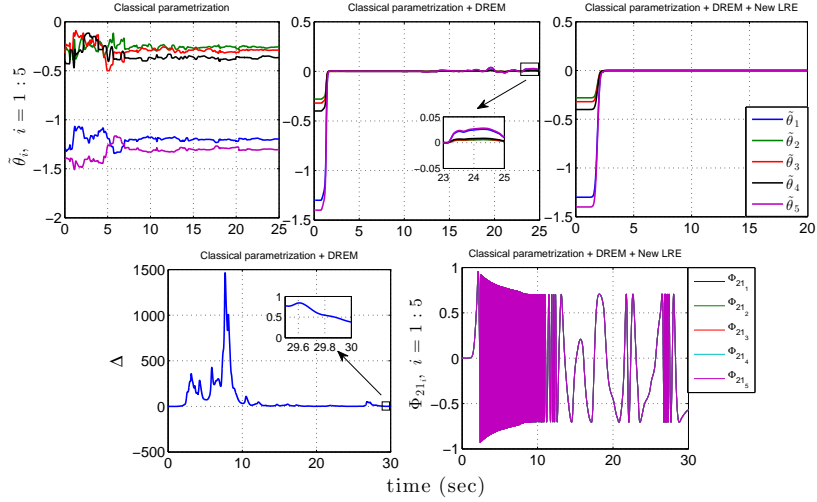}  
  \caption{Transient behavior of the signals $\tilde \theta_i$ using $\tau_c$.}
    \label{fig3}
\end{figure}

In the next series of simulations we used the power balance equation-based parameterization, with the adaptation gains set  as   $\Gamma=100I_5$ for the estimator \eqref{paa} and $\gamma_i=100$ for the DREM ones. The result of the simulations is depicted in Figs. \ref{fig4}-\ref{fig6}, which shows a similar scenario as the classical parameterization. One notable difference between the two parameterizations is that with the new one the excitation of $\Phi_{21}$ is lost. Also, notice that using only DREM---without the new LRE---the parameters do not converge, revealing the critical importance of this modification.

\begin{figure}[htp]
\centering
  \includegraphics[width=.83\textwidth]{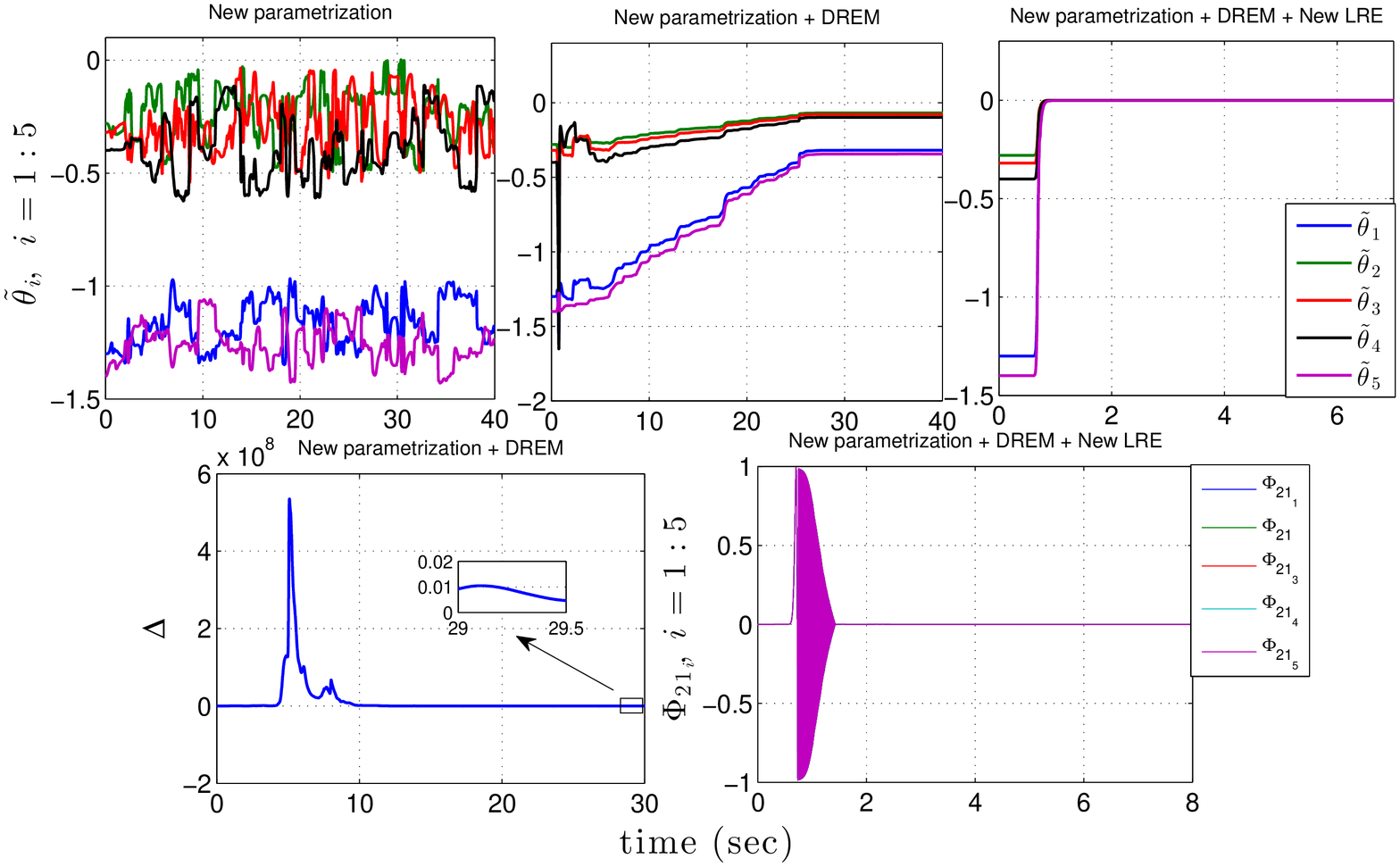}  
  \caption{Transient behavior of the signals $\tilde \theta_i$, $\Delta$ and $\Phi_{21_i}$ using $\tau_a$.}
  \label{fig4}
\end{figure}

\begin{figure}[htp]
\centering
  \includegraphics[width=.83\textwidth]{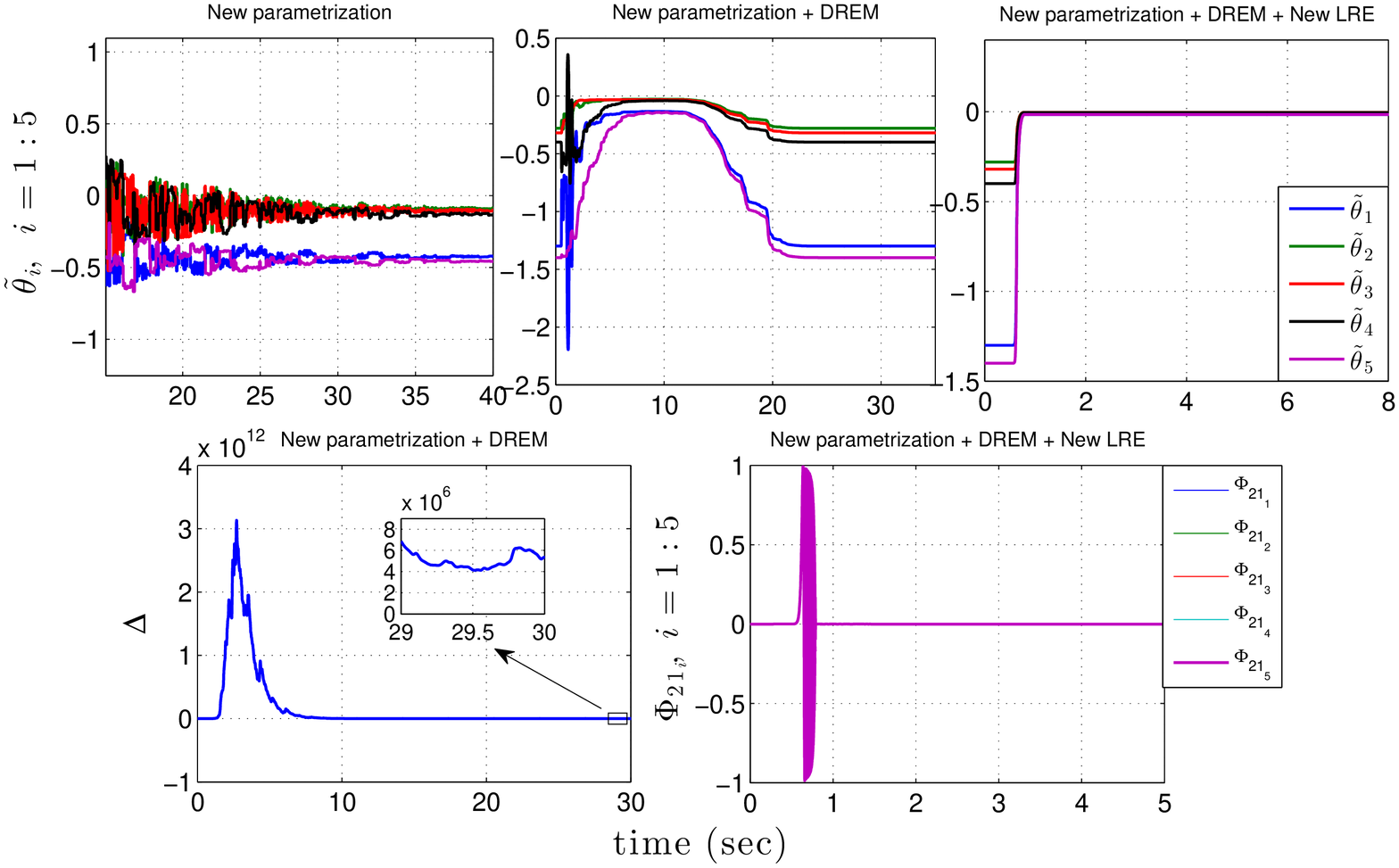}  
  \caption{Transient behavior of the signals $\tilde \theta_i$, $\Delta$ and $\Phi_{21_i}$ using $\tau_b$.}
  \label{fig5}
\end{figure}

\begin{figure}[htp]
\centering
  \includegraphics[width=0.83\textwidth]{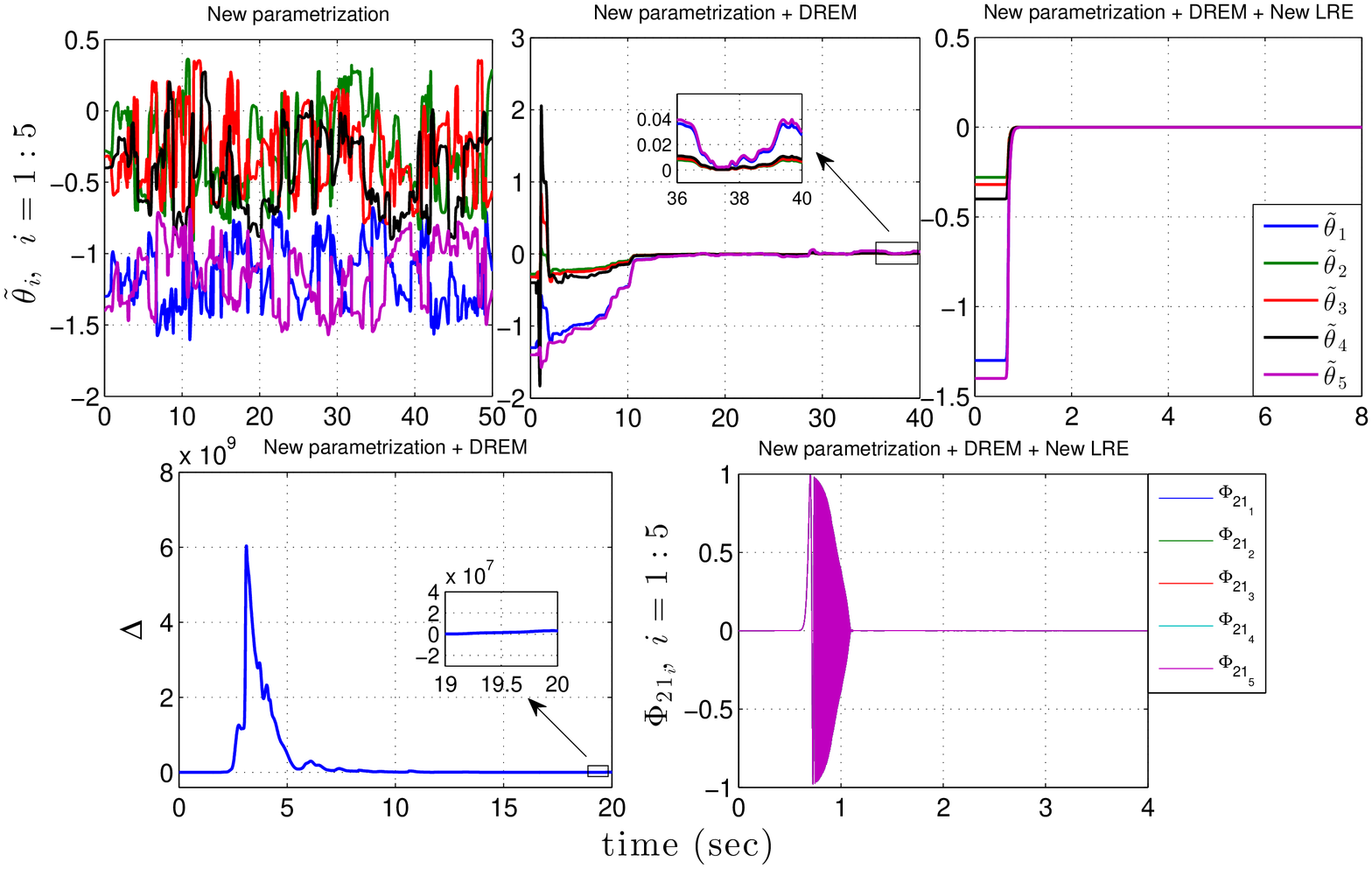}  
  \caption{Transient behavior of the signals $\tilde \theta_i$,  $\Delta$ and $\Phi_{21_i}$  using $\tau_c$.}
  \label{fig6}
\end{figure}

\subsubsection*{Adaptive control} Our last series of simulations pertains to the implementation of the Slotine-Li adaptive controller of Remark \ref{rem5} in regulation and tracking. In both cases we selected the control gains $K_1=7 I_2$ and $K_2=4I_2$ and the adaptation gain $\Gamma= 25I_5$ and $\gamma_i= 10$ for the regulation case and $\gamma_i=25$ for the tracking problem.  We  propose as desired equilibrium point $q_\star=\col(.2\pi,\, 0.3\pi)$, and for the tracking problem 
$$
q_\star=\left[ \begin{array}{c}0.4 \pi \sin(0.4 t)+0.3 \pi \sin(0.3t) + 0.2 \pi \\ 0.3 \pi \cos( 0.3 t)-0.1 \pi \cos(0.5 t) + 0.3\pi \end{array} \right].
$$
The result of the simulations is depicted in Figs. \ref{fig7}-\ref{fig11} for the case of regulation and Figs. \ref{fig12}-\ref{fig16} for tracking, where we also show the behavior of the input signal. From Fig. \ref{fig7} we see that for the classical parameterization the regulation objective is achieved without parameter convergence, while the new parameterization is unsuccessful.  This is consistent with the well-known fact that there is no unique set of controllers gains that achieve  the regulation objective. In Fig. \ref{fig8} we observe that the addition of DREM to the classical parameterization corrects the lack of parameter convergence for the classical parameterization but it is of no use for the new one---for both cases the excitation is lost, as shown in Fig. \ref{fig9}. The performance for both parametrizations is drastically improved with the use of the new LRE as shown in  Fig. \ref{fig10}, ensuring in both cases PE as depicted in Fig. \ref{fig11}.  

The same pattern as in the case of regulation is observed for tracking: Fig. \ref{fig12} shows the good behavior of the old parameterization while the new one does not achieve neither the control objective nor parameter convergence.  Again, the addition of the new LRE significantly improves the behavior for both parameterizations underscoring (see Figs. \ref{fig13}-\ref{fig15}), once again, the critical importance of this modification.

\begin{figure}[htp]
\centering
  \includegraphics[width=0.83\textwidth]{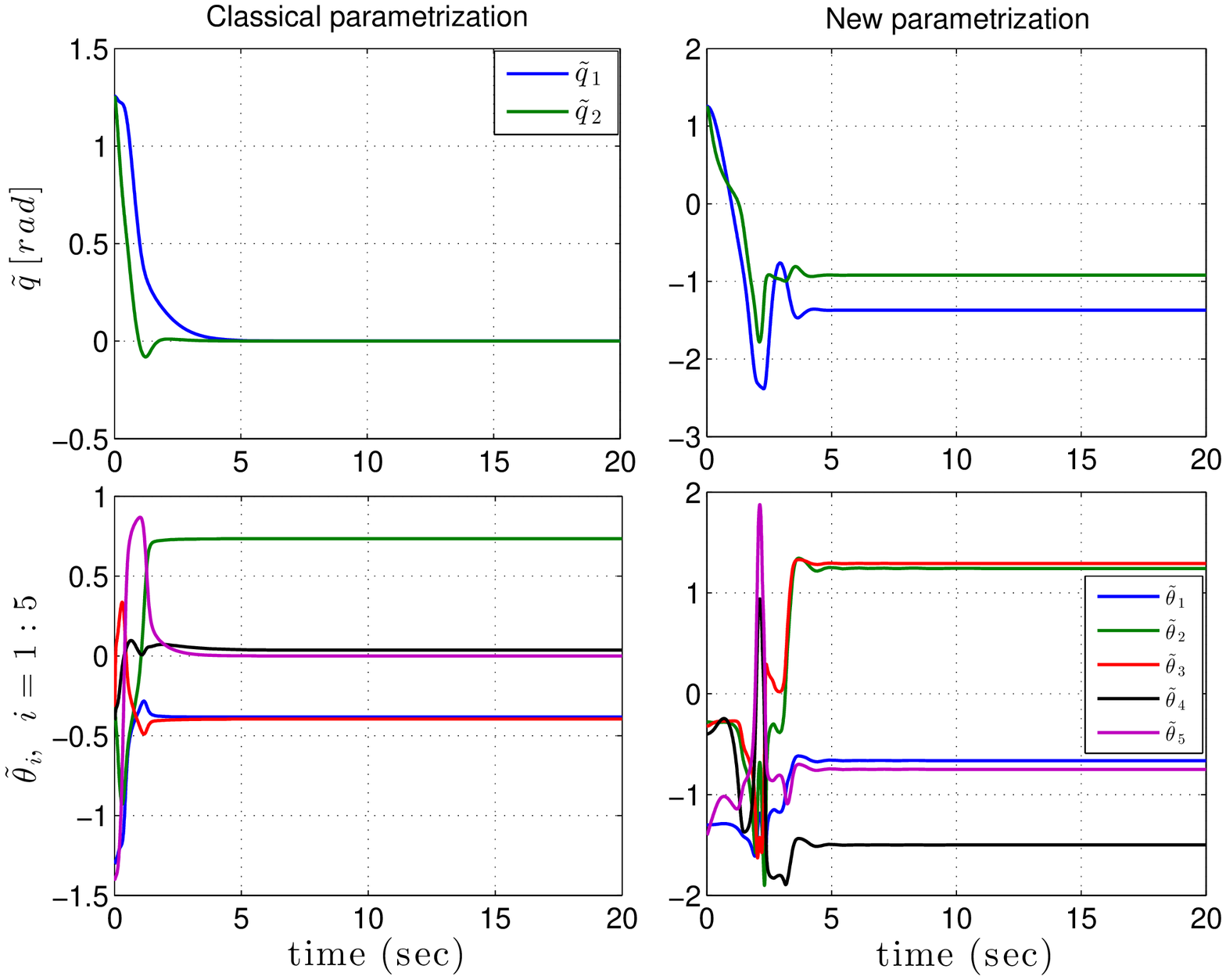}  
  \caption{Transient behavior of the signals $\tilde q_i$  and $\tilde \theta_i$}
  \label{fig7}
\end{figure}

\begin{figure}[htp]
\centering
  \includegraphics[width=0.83\textwidth]{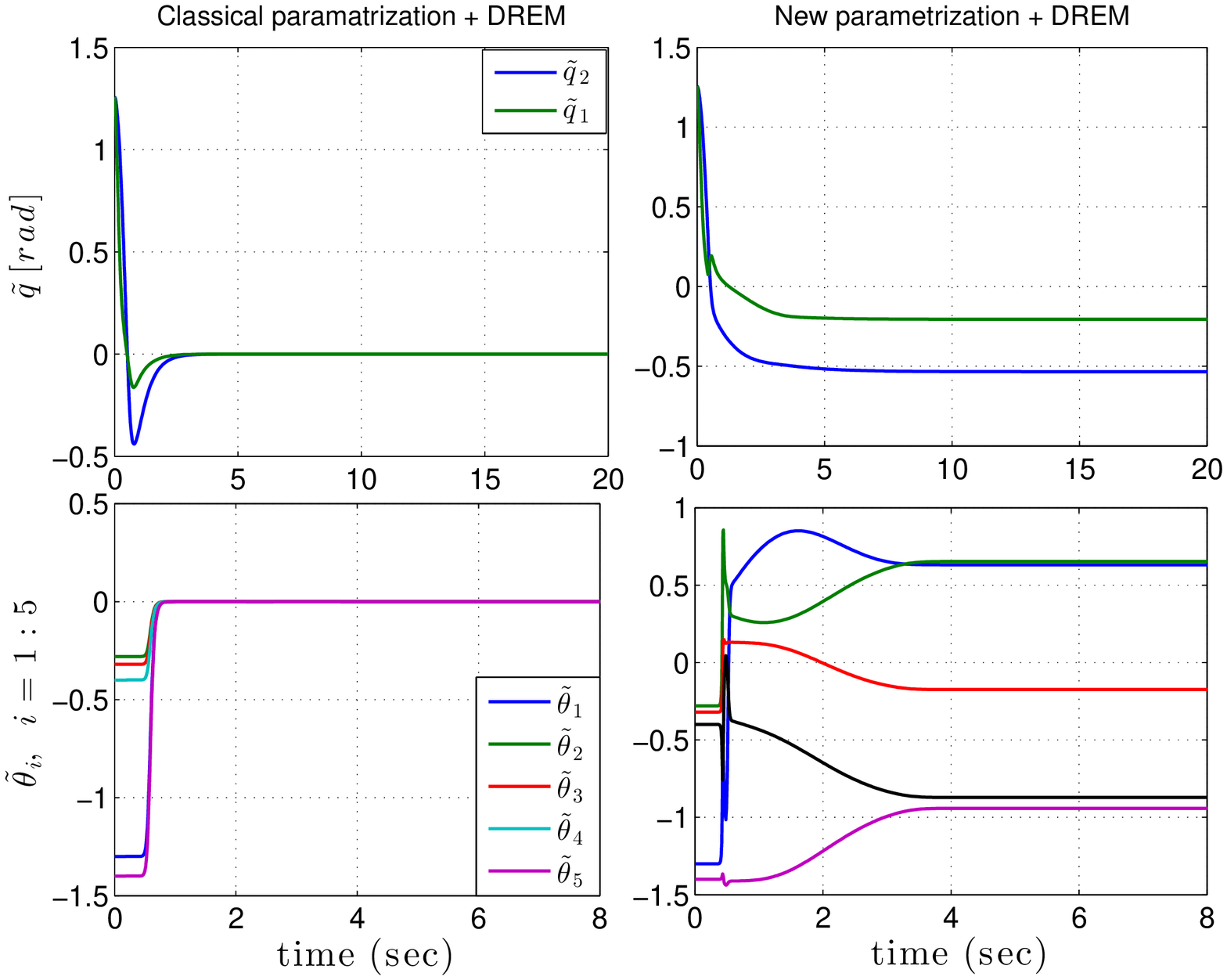}  
  \caption{Transient behavior of the signals $\tilde q_i$  and $\tilde \theta_i$}
\label{fig8}
\end{figure}
\begin{figure}[htp]
\centering
  \includegraphics[width=0.83\textwidth]{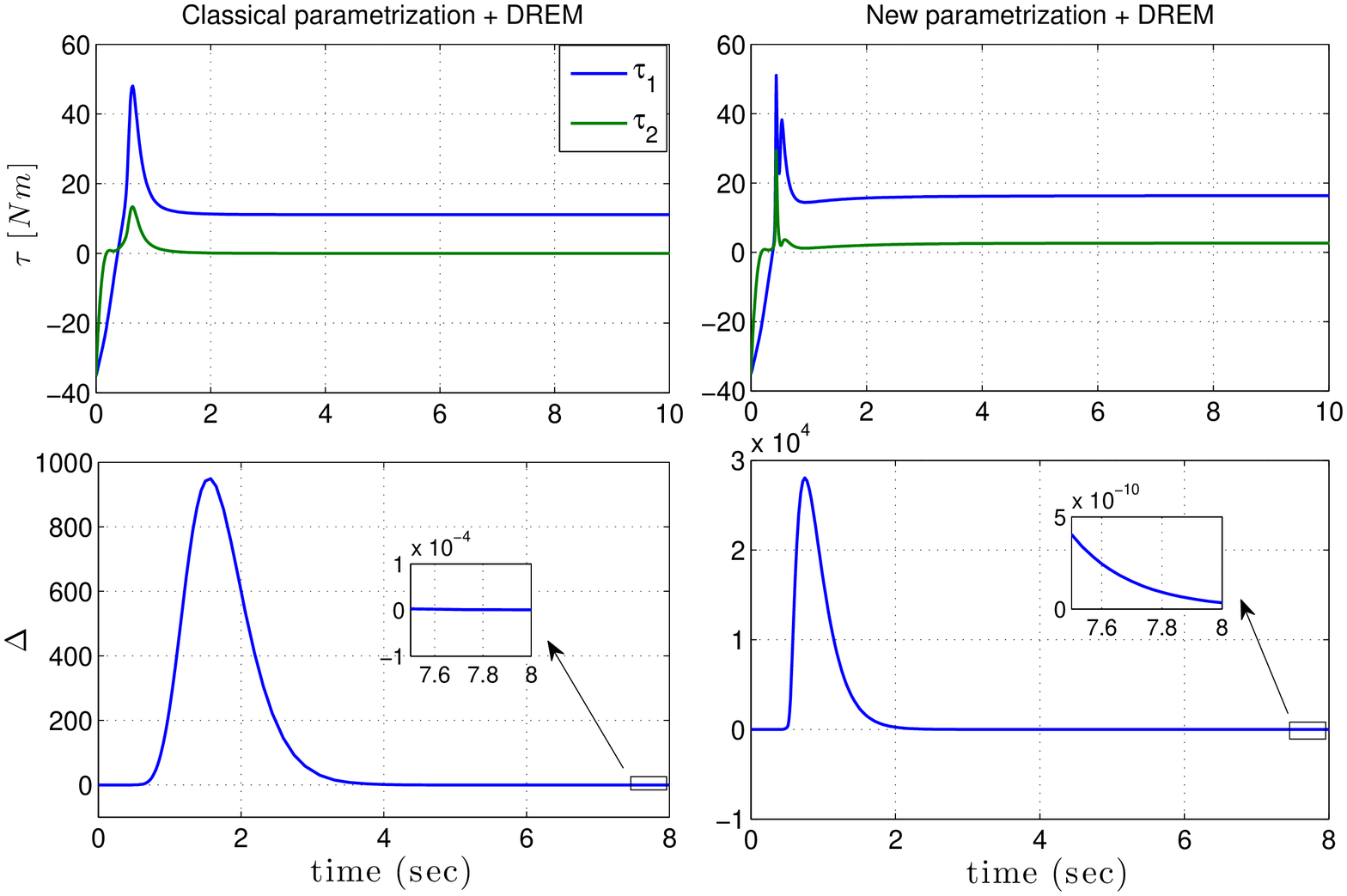}  
  \caption{Transient behavior of the input control $\tau$ and the signal $\Delta$ }
 \label{fig9}
\end{figure}
\begin{figure}[htp]
\centering
  \includegraphics[width=0.83\textwidth]{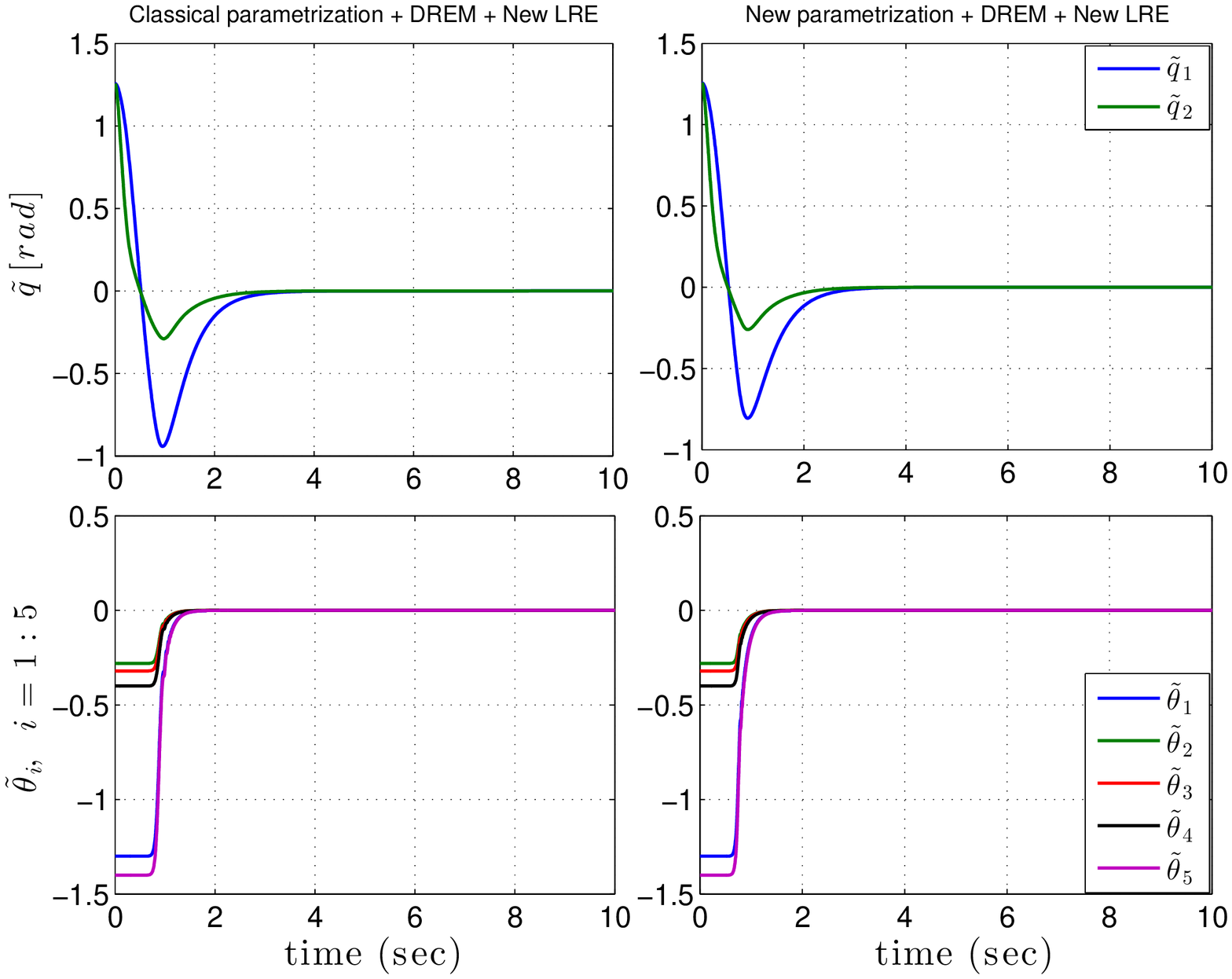}  
  \caption{Transient behavior of the signals $\tilde q_i$  and $\tilde \theta_i$}
 \label{fig10}
\end{figure}
\begin{figure}[htp]
\centering
  \includegraphics[width=0.83\textwidth]{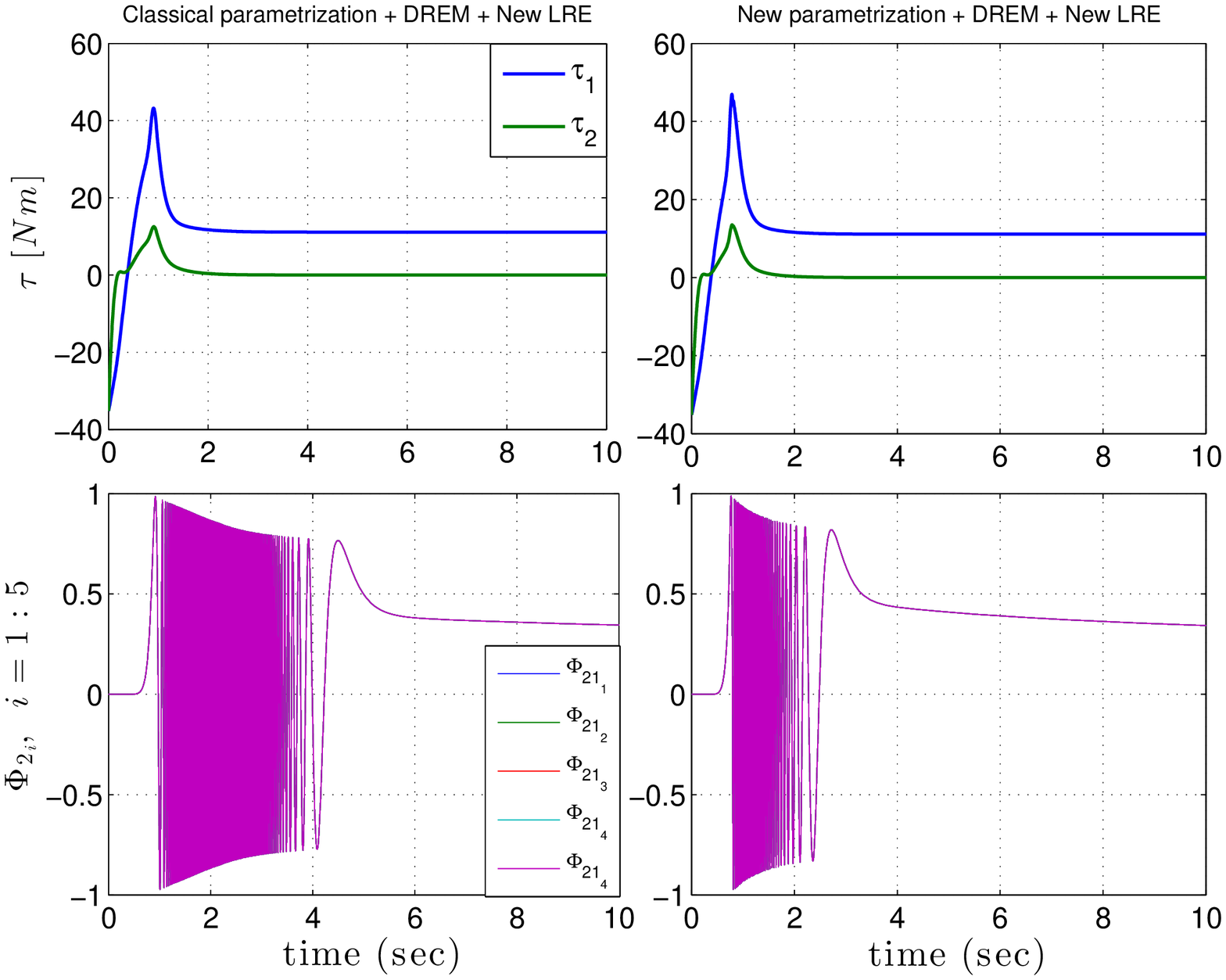}  
  \caption{Transient behavior of the input control $\tau$   and exciting signal $\Phi_{21_i}$}
 \label{fig11}
\end{figure}


\begin{figure}[htp]
\centering
  \includegraphics[width=0.83\textwidth]{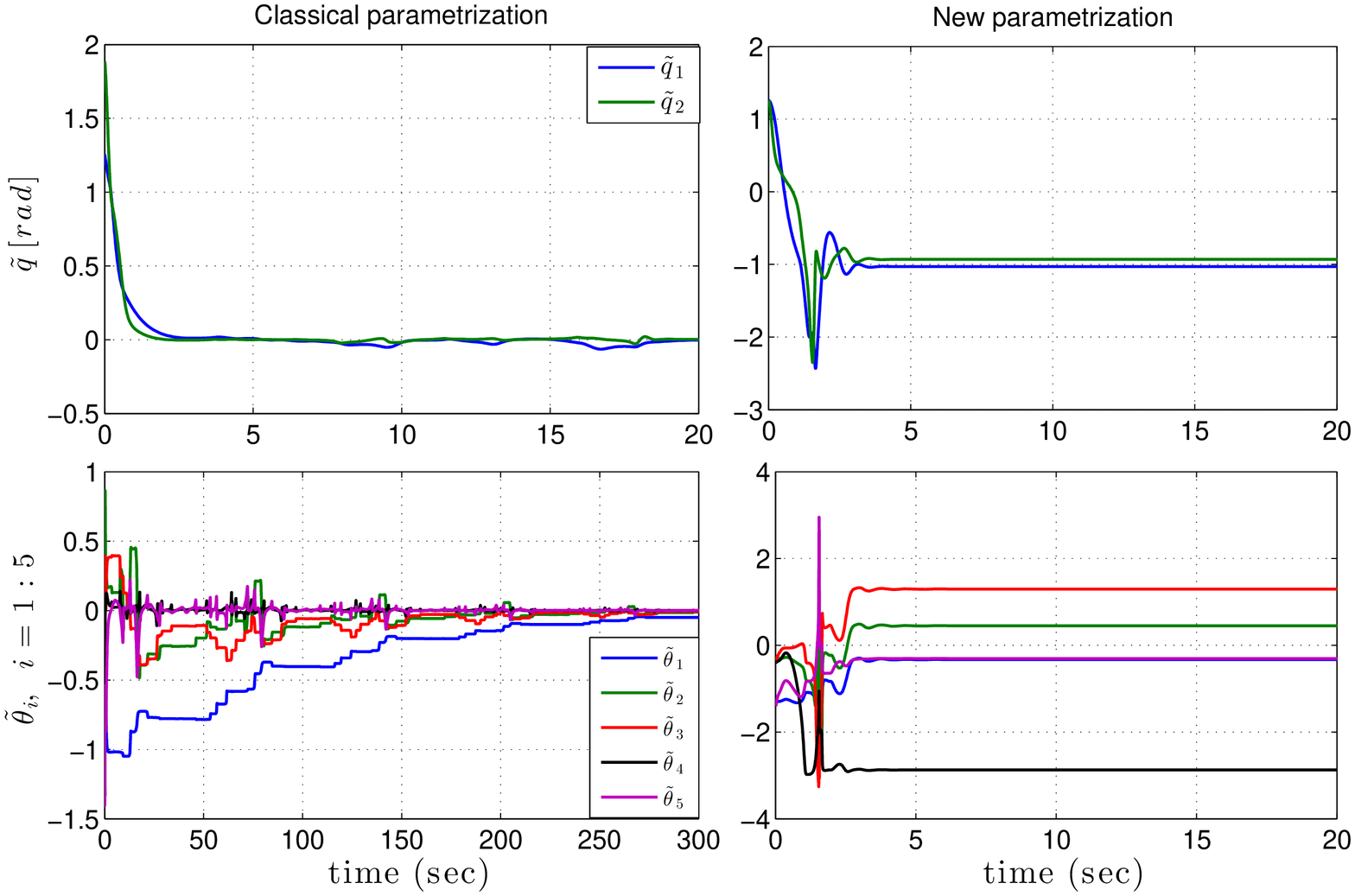}  
  \caption{Transient behavior of the signals $\tilde q_i$  and $\tilde \theta_i$.}
  \label{fig12}
\end{figure}

\begin{figure}[htp]
\centering
  \includegraphics[width=0.83\textwidth]{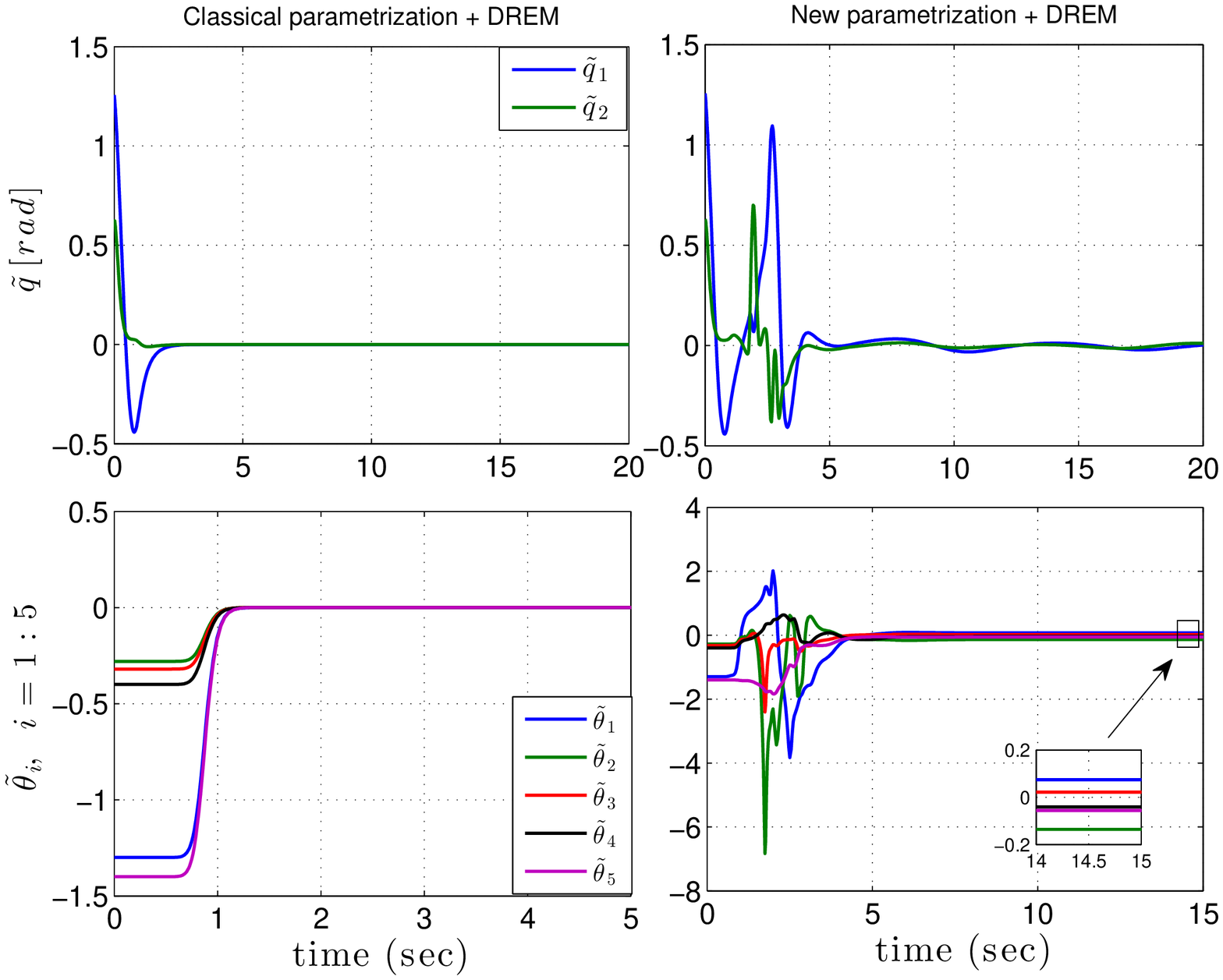}  
  \caption{Transient behavior of the signals $\tilde q_i$  and $\tilde \theta_i$}
\label{fig13}
\end{figure}
\begin{figure}[htp]
\centering
  \includegraphics[width=0.83\textwidth]{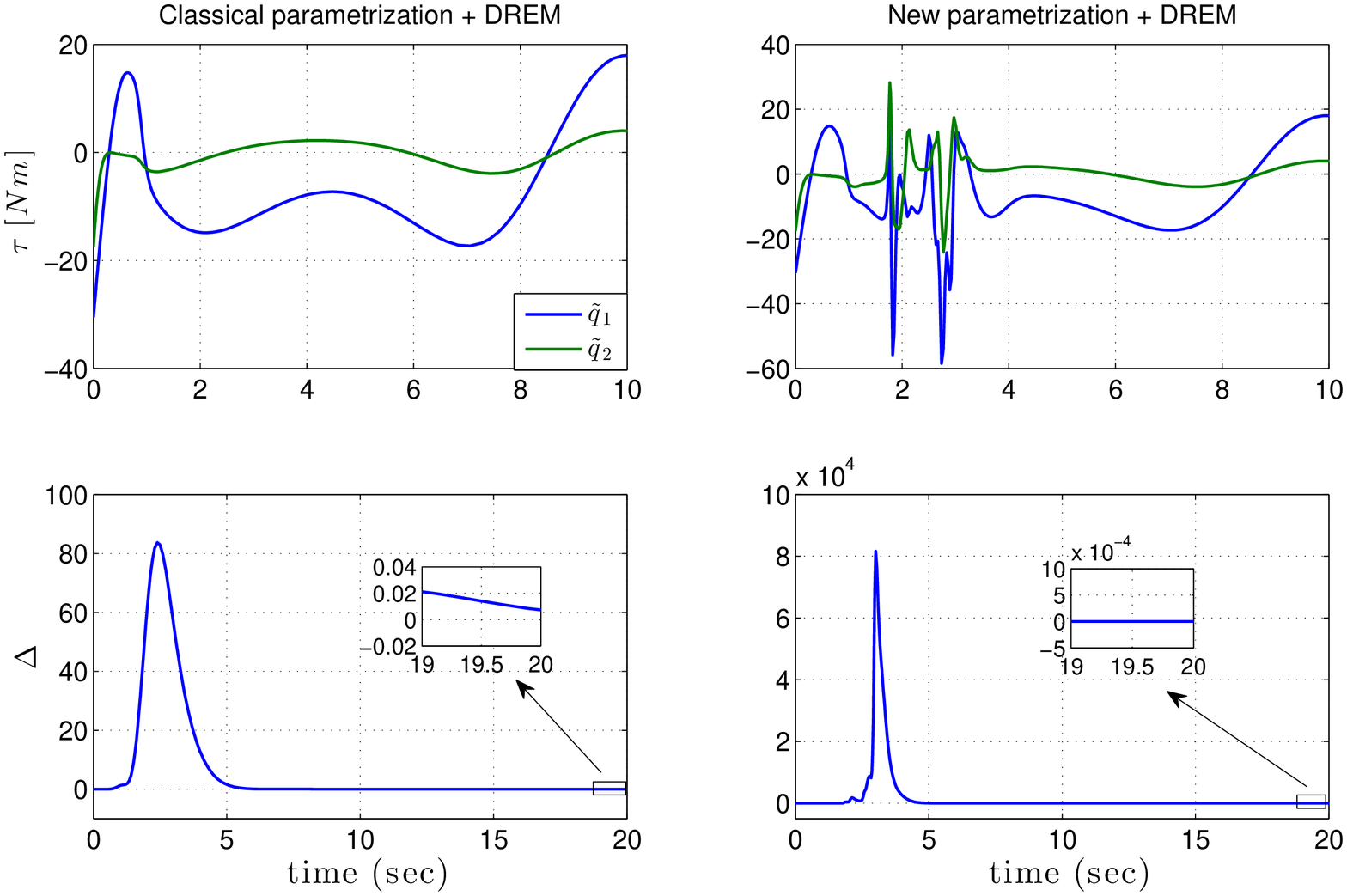}  
  \caption{Transient behavior of the input control $\tau$ and the signal $\Delta$ }
 \label{fig14}
\end{figure}
\begin{figure}[htp]
\centering
  \includegraphics[width=0.83\textwidth]{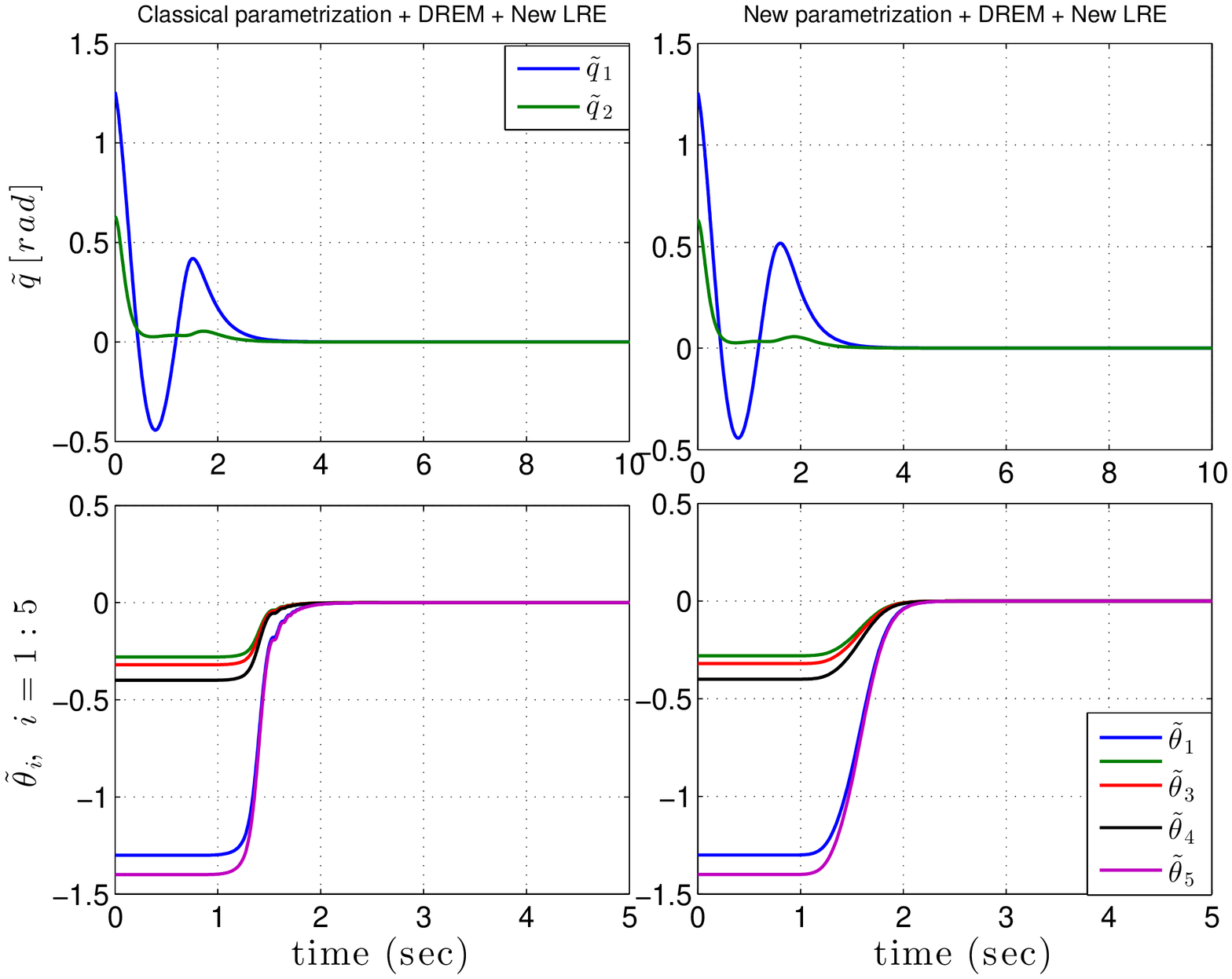}  
  \caption{Transient behavior of the signals $\tilde q_i$  and $\tilde \theta_i$}
 \label{fig15}
\end{figure}
\begin{figure}[htp]
\centering
  \includegraphics[width=0.83\textwidth]{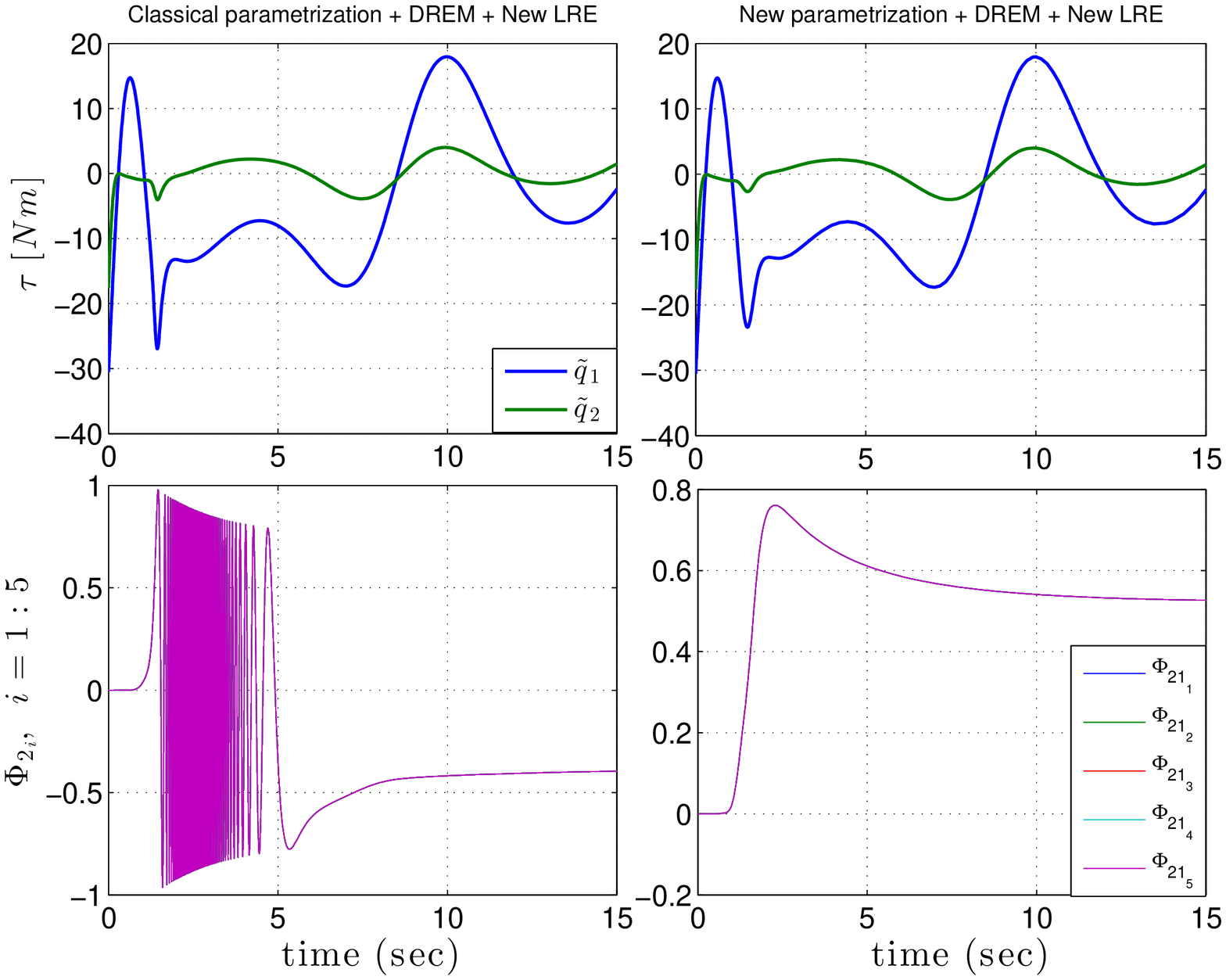}  
  \caption{Transient behavior of the  input control $\tau$   and exciting signal $\Phi_{21_i}$}
  \label{fig16}
\end{figure}

Summarizing: the simulation results confirm our claim that the new parameterization---which is of interest for its reduced computational complexity---is not applicable without DREM {\em and} the new LRE. But it becomes a feasible practical solution including these two modifications. It should be higlighted that, even though the addition of DREM and the new LRE entails additional computations, they are of much smaller magnitude than the ones required for the implementation of the classical parameterization. 

%
\section{Conclusions and Future Research}
\lab{sec5}
%
We have proven, for the first time, that the significantly simpler parameterization obtained from the power balance equation of EL systems can be used for their identification and adaptive control even in a scenario with insufficient excitation. The key  modification that is required is the use of new LREs, which are generated following the procedure proposed in \citep{BOBetal}. We believe this contribution paves the way for a wider utilization of these advanced control techniques in critical areas like robotics. 

The main result of the paper can be {\em verbatim} extended to the very broad class of {\em passive} nonlinear systems with linearly parameterized storage function. Indeed, in this case the system
\begin{align}
\dot x &= f(x,u_p)\\
y_p &=h(x,u_p),
\end{align}
with $x \in \rea^n$, $u_p \in \rea^m$ and  $y_p \in \rea^m$, verifies the relation
$$
\dot \calh = u_p^\top y_p,
$$
where the storage function $\calh:\rea^n \to \rea_+$ may be expressed as
$$
\calh(x)=\phi^\top(x)\theta + b(x),
$$  
with known functions $\phi:\rea^n \to \rea^q$ and $b:\rea^n \to \rea$, and $\theta\in \rea^q$ a vector of unknown parameters. It is clear that Proposition \ref{pro1} applies immediately to this case with the definitions
\begin{align}
y &:= H(p)[u_p^\top y_p]- pH(p)[ b(x)]\\
\Omega &:=pH(p)[\phi(x)].
\end{align}
Current research is under way to apply this result to electro-mechanical systems. 
\\
Another immediate extension is to the case of separable {\em nonlinear} parameterizations, that is the case when the regression equation has the form
$$
y(t)= \phi(t)^\top  {\cal G}(\theta),
$$
with known functions ${\cal G}: \rea^{p} \to \rea^{q}$, with $p<q$,  of the unknown parameters $\theta \in \rea^p$.  See \citep{ORTetalaut21} for additional details.

\section*{Acknowledgements}

The second author is grateful to Jean-Jacques Slotine for  insightful remarks about the parameterization of the power balance equation and to Mark Spong for bringing to his attention the work of Gautier and Khalil.
This paper is partly supported by the Ministry of Education and Science of Russian Federation (14.Z50.31. 0031, goszadanie no. 8.8885.2017/8.9), NSFC (61473183, U1509211).

\section{Appendix}
\subsection{Background Material}
\lab{appa}
In this appendix we present the following preliminary results.
\begenu
\item[{\bf B1}] Derivation of Kreisselmeier's  regressor extension \citep[Proposition 3]{ORTNIKGER} with the DREM estimator \citep[Proposition 1]{ARAetaltac}.
\item[{\bf B2}]  Generation of new LRE \citep{BOBetal} and excitation injection via energy pumping-and-damping \citep{YIetal}.
\item[{\bf B3}]  Properties of the standard gradient estimator for the new LRE.
\endenu

 \begin{proposition}
\lab{pro4}\em
Consider the LRE
$$
\boldsymbol{Y}=\boldsymbol{\Omega} \boldsymbol\theta
$$
where $\boldsymbol{Y}(t) \in \rea^n,\;\boldsymbol{\Omega}(t) \in \rea^{n \times q}$ are {\em measurable} signals and $\boldsymbol{\theta} \in \rea^q$ is a constant vector of {\em unknown} parameters. Fix $\lambda >0$ and define the signals
\begin{align}
\nonumber
\dot {\boldsymbol Z} &=-\lambda \boldsymbol Z + \boldsymbol \Omega^\top \boldsymbol Y\\
\nonumber
\dot {\boldsymbol \Psi} &=-\lambda \boldsymbol \Psi + \boldsymbol \Omega^\top  \boldsymbol \Omega\\
\nonumber
\boldsymbol \caly &= \adj\{\boldsymbol \Psi\}\boldsymbol Z\\
\lab{eq1}
\boldsymbol \Delta &=\det\{\boldsymbol \Psi\}.
\end{align}
The $q$ scalar LRE
\begequ
\lab{scanlpre}
\boldsymbol \caly_i=\boldsymbol \Delta  \boldsymbol \theta_i,\;i \in \bar q:=\{1,2,\dots,q\}.
\endequ
hold.
\end{proposition}

\begin{proposition} \em
\lab{pro5}
Consider the scalar LREs\footnote{To simplify the notation we omit the subindex $i$ in the proposition.}  \eqref{scanlpre} with $\boldsymbol \Delta$ interval exciting,  \citep{KRERIE}, that is such that 
$$
\int_{0}^{t_{c}}\boldsymbol \Delta^2(\tau) d\tau \geq \delta.
$$
for some $t_{c}>0$ and $\delta>0$.

Define the dynamic extension
\begin{align}
\nonumber
	\dot {\boldsymbol z} & =\boldsymbol  u_{2}\boldsymbol \caly+\boldsymbol  u_{3}\boldsymbol  z,\;\boldsymbol z(0)=0 \\
\nonumber
	 \dot{ \boldsymbol \xi} & =\boldsymbol  A(t)\boldsymbol  \xi +\boldsymbol  b(t),\;\boldsymbol  \xi(0)=\col(0,0) \\
\lab{eq2}
\dot {\boldsymbol   \Phi}& =\boldsymbol  A(t)\boldsymbol  \Phi,\; \boldsymbol \Phi(0)=I_2,
\end{align}
where
$$
\boldsymbol A(t) :=\left[ \begin{array}{cc}0 & \boldsymbol u_{1}(t) \\\boldsymbol  u_{2}(t) \boldsymbol \Delta(t) &\boldsymbol  u_{3}(t)\end{array} \right],\; \boldsymbol b(t):=\left[\begin{array}{c}-\boldsymbol u_{1}(t) \boldsymbol z(t)\\ 0 \end{array} \right],
$$
with $\boldsymbol u_{1}(t),\boldsymbol u_{2}(t),\boldsymbol u_{3}(t) \in \rea$ {\em arbitrary} signals.
\begenu
\item[{\bf P1}] The new LRE
\begin{equation}
\label{newlre}
Y  = { \boldsymbol  \Phi}_{21} \boldsymbol \theta, 
\end{equation}
holds with
$$
{Y}:=\boldsymbol z- \boldsymbol \xi_{2}
$$
and ${\boldsymbol   \Phi}_{21}(t) \in \rea$ the $(\cdot)_{21}$ element of the matrix $\boldsymbol \Phi$.

\item[{\bf P2}]  Define the signals
\begin{align}
\nonumber
\boldsymbol u_{1} & =- \boldsymbol \alpha \boldsymbol \Delta\\
\nonumber
\boldsymbol u_{2} & =\boldsymbol \alpha\\
\lab{eq3}
\boldsymbol u_{3} &= -\tilde V(\boldsymbol  \Phi_{11},\boldsymbol \Phi_{21}),
\end{align}
where  $0<\boldsymbol \beta< \hal$, $\boldsymbol \alpha(t) \in \rea$ is a bounded signal such that $\boldsymbol \alpha(t) \boldsymbol \Delta(t) \not \equiv 0$,
\begequ
\lab{conalpdel}
\boldsymbol \alpha(t) \boldsymbol \Delta(t) \in \call_1,
\endequ 
and
$$
\tilde V(\boldsymbol  \Phi_{11},\boldsymbol \Phi_{21})  := \hal (\boldsymbol \Phi_{11}^2 +\boldsymbol  \Phi_{21}^2) -\boldsymbol  \beta.
$$
The full state of the LRE generator---that is, $\boldsymbol \Phi,\boldsymbol \xi$ and $\boldsymbol z$---is bounded and either $\boldsymbol \Phi_{21}\not \in \call_2$ or 
\begequ
\lab{bouphi}
\boldsymbol  \Phi^2_{11}(t)+\boldsymbol \Phi^2_{21}(t) \geq 2 \boldsymbol  \beta+\boldsymbol \varepsilon, \quad \forall t \ge 0,
\endequ
for some (sufficiently small) $\boldsymbol \varepsilon>0$.
\endenu
\end{proposition}

 \begin{proposition}
\lab{pro6}\em
Consider the scalar LRE  \eqref{newlre} of  Proposition \ref{pro5} with the gradient estimator
\begequ
\label{eq4}
\dot{\hat{\boldsymbol \theta}} = \boldsymbol \gamma  \frac{ { \boldsymbol  \Phi}_{21}}{1+ { \boldsymbol  \Phi}_{21}^2} \left( Y -  { \boldsymbol  \Phi}_{21}\, \hat {\boldsymbol \theta}\right),
\endequ
with $\boldsymbol \gamma>0$.  The following implication is true:
$$
\liminf \boldsymbol \Phi_{11}(t) \neq \sqrt{2\beta} \; \Rightarrow \;\liminf \hat{\boldsymbol \theta}(t) =\boldsymbol \theta.
$$
\end{proposition}

\begin{proof}
The proof follows immediately noting that, on one hand, $\boldsymbol \Phi_{21}\not \in \call_2$ ensures parameter convergence \citep[Proposition 1]{ARAetaltac}. On the other hand, in the light of \eqref{bouphi}, the condition $\liminf \boldsymbol \Phi_{11}(t) \neq \sqrt{2\beta}$ implies that $\boldsymbol  \Phi_{21}$ is PE, yielding exponential convergence.
\end{proof}


\end{document}